\documentclass[a4paper,11pt]{article}
\title{Substitution and $\chi$-Boundedness}
\author{Maria Chudnovsky\thanks{Columbia University, New York, NY 10027, USA. E-mail: mchudnov@columbia.edu. Partially supported by NSF grants DMS-0758364 and DMS-1001091.}~, Irena Penev\thanks{Park Tudor School, Indianapolis, IN 46240, USA. E-mail: ip2158@caa.columbia.edu. This research was conducted while the author was a graduate student at Columbia University.}~, 
Alex Scott\thanks{Mathematical Institute, University of Oxford, 24-29 St Giles', Oxford OX1 3LB, UK.
E-mail: scott@maths.ox.ac.uk.}~, Nicolas Trotignon\thanks{CNRS, LIP, ENS Lyon, Universit\'e de Lyon, INRIA\ email:
    nicolas.trotignon@ens-lyon.fr. Partially supported by the French \emph{Agence Nationale de la
Recherche} under reference \textsc{anr-10-jcjc-Heredia.}}} 
\date{February 1, 2013} 
\usepackage{amsthm}
\usepackage{amssymb}
\usepackage{amsbsy}
\usepackage{mathrsfs} 
\newtheorem{theorem}{Theorem}[section] 
\newtheorem{lemma}[theorem]{Lemma} 
\newtheorem{prop}[theorem]{Proposition} 
\newtheorem{cor}[theorem]{Corollary} 
\newtheorem{question}{Question} 
\newtheorem*{gluing bound}{Lemma \ref{gluing bound}} 
\newtheorem*{small cutsets}{Theorem \ref{small cutsets}} 
\newtheorem*{gluing combination}{Proposition \ref{gluing combination}} 
\newtheorem*{closure}{Proposition \ref{closure}}

\begin{document}
\maketitle 

\begin{abstract} 
\noindent 
A class $\mathcal{G}$ of graphs is said to be {\em $\chi$-bounded} if there is a function $f:\mathbb{N} \rightarrow \mathbb{R}$ such that for all $G \in \mathcal{G}$ and all induced subgraphs $H$ of $G$, $\chi(H) \leq f(\omega(H))$. In this paper, we show that if $\mathcal{G}$ is a $\chi$-bounded class, then so is the closure of $\mathcal{G}$ under any one of the following three operations: substitution, gluing along a clique, and gluing along a bounded number of vertices. Furthermore, if $\mathcal{G}$ is $\chi$-bounded by a polynomial (respectively: exponential) function, then the closure of $\mathcal{G}$ under substitution is also $\chi$-bounded by some polynomial (respectively: exponential) function. In addition, we show that if $\mathcal{G}$ is a $\chi$-bounded class, then the closure of $\mathcal{G}$ under the operations of gluing along a clique and gluing along a bounded number of vertices together is also $\chi$-bounded, as is the closure of $\mathcal{G}$ under the operations of substitution and gluing along a clique together. 
\end{abstract}

\section{Introduction} 

All graphs in this paper are simple and finite (possibly empty). A
{\em clique} in a graph $G$ is a set of pairwise adjacent vertices of
$G$. The {\em clique number} of a graph $G$ (i.e.\ the maximum number
of vertices in a clique in $G$) is denoted by $\omega(G)$, and the
{\em chromatic number} of $G$ (i.e.\ the smallest number of colors
needed to properly color $G$) is denoted by $\chi(G)$. A class
$\mathcal{G}$ of graphs is said to be {\em $\chi$-bounded} provided
that there exists a function $f:\mathbb{N} \rightarrow
\mathbb{R}$ such that for all $G \in \mathcal{G}$, and all (possibly
empty) induced subgraphs $H$ of $G$, $\chi(H) \leq f(\omega(H))$;
under such circumstances, we say that the class $\mathcal{G}$ is {\em
  $\chi$-bounded by $f$}, and that $f$ is a {\em $\chi$-bounding
  function} for $\mathcal{G}$. (We observe that there is no loss in
generality in assuming that the function $f$ is non-decreasing;
indeed, if $G$ is $\chi$-bounded by a function $f:\mathbb{N} 
\rightarrow \mathbb{R}$, then it is also $\chi$-bounded by the
non-decreasing function $g:\mathbb{N} \rightarrow
\mathbb{R}$ given by $n \mapsto \max\{f(0),...,f(n)\}$.) A well-known
example of a $\chi$-bounded class is the class of {\em perfect}
graphs, that is, graphs $G$ that satisfy the property that for every
induced subgraph $H$ of $G$, $\chi(H) = \omega(H)$. Clearly, the class
of perfect graphs is $\chi$-bounded by the identity function.
\\
\\
A class of graphs is said to be {\em hereditary} if it is closed under
isomorphism and induced subgraphs. (In particular, every non-empty 
hereditary class contains the empty graph.) Note that if
$\mathcal{G}$ is a hereditary class, then $\mathcal{G}$ is
$\chi$-bounded if and only if there exists a function $f:\mathbb{N} 
\rightarrow \mathbb{R}$ such that $\chi(G) \leq
f(\omega(G))$ for all $G \in \mathcal{G}$, and under such
circumstances, $f$ is a $\chi$-bounding function for
$\mathcal{G}$. $\chi$-boundedness has mostly been studied in the
context of classes obtained by forbidding certain families of graphs
as induced subgraphs, and all such classes are easily seen to be
hereditary.
\\
\\
Several theorems and conjectures state that some classes obtained by
forbidding certain families of graphs are $\chi$-bounded. A well known
example is the strong perfect graph theorem (see~\cite{SPGT}), which
states that a graph $G$ is perfect if and only if it is Berge (a graph
$G$ is {\em Berge} if neither $G$ nor its complement contains an
induced odd cycle of length at least five). Many others exist; see for
instance~\cite{gyarfas:perfect, gst, kp, kz, scott:tree, s2}. Some of
these theorems can be proven (and some of these conjectures might be
proven) by using the following approach: show that all graphs of the
class are either in some well understood basic class, or can be cut
into pieces by some appropriate operation (or decomposition) that in a
sense ``preserves'' the bound on the chromatic number that is to be
proven. This approach (let us call it the \emph{structural approach})
is far from being the only one: some deep $\chi$-boundedness results
have been proven by using extremal (e.g.~Ramsey-theoretic) or
probabilistic arguments. But the structural approach seems to be very
effective for finding optimal $\chi$-bounding functions, and in some
situations, it is the only known approach. For instance, the
structural proof of the strong perfect graph theorem~\cite{SPGT}
(together with a shortening, see~\cite{chudnovsky.seymour:even}) is
the only proof currently available for the $\chi$-boundedness of Berge
graphs.
\\
\\
All this raises the following question: what operations preserve
$\chi$-boundedness? A great deal of research has been devoted to a
similar question: what are the operations that preserve perfection
(see, for instance,~\cite{livre:perfectgraphs})? Historically, the
main operations that have been considered are clique
cutsets~\cite{gallai:triangule}, substitutions~\cite{Lovasz},
amalgams~\cite{burlet.fonlupt:meynieltop},
2-joins~\cite{cornuejols.cunningham:2join}, homogeneous
pairs~\cite{chvatal.sbihi:bullfree}, star cutsets and skew
partitions~\cite{chvatal:starcutset}. Note that the word ``operation''
is perhaps not precise here, because the star cutset and the skew
partition are not really operations that allow one to build bigger
graphs by gluing pieces together. Nevertheless, they are very
important for perfect graphs.
\\
\\
These same operations can also be studied in the context of
$\chi$-boundedness. In this paper, we consider two of these
operations: clique cutsets and substitution (a {\em clique cutset} in
a graph $G$ is a clique in $G$ whose deletion from $G$ yields a
disconnected graph; substitution is defined in subsection
\ref{subsection:definitions}). While a clique cutset is a
decomposition, rather than an operation, it is easy to see that it
corresponds to the operation of ``gluing along a clique,'' which we
define in subsection \ref{subsection:definitions}. We also consider
one additional operation: gluing along a bounded number of vertices
(see subsection \ref{subsection:definitions} for a formal definition).
\\
\\
The paper is organized as follows. In section
\ref{section:substitution}, we show that the closure of a
$\chi$-bounded class under substitution is again $\chi$-bounded (see Theorem 
\ref{general bound}), and we also examine the effects of substitution
on $\chi$-bounding functions. In particular, we show the following: if
a class $\mathcal{G}$ is $\chi$-bounded by a polynomial function $P$,
then there exists a polynomial function $Q$ such that the closure of
$\mathcal{G}$ under substitution is $\chi$-bounded by $Q$ (see
Theorem \ref{poly}). Interestingly, the degree of $Q$ cannot be bounded by any
function of the degree of $P$ (see Proposition \ref{fractional}). Further, we
prove that if a class $\mathcal{G}$ is $\chi$-bounded by an
exponential function, then the closure of $\mathcal{G}$ under
substitution is also $\chi$-bounded by some exponential function (see
Theorem \ref{exp}).
\\
\\
In section \ref{section:gluing}, we turn to the two gluing
operations. It is easy to show that the closure of a $\chi$-bounded
class under gluing along a clique is $\chi$-bounded (see Proposition 
\ref{gluing clique only}). Next, we show that the closure of a $\chi$-bounded
class under gluing along at most $k$ vertices (where $k$ is a fixed
positive integer) is $\chi$-bounded (see Theorem \ref{small cutsets}). We note that this answers an open question from~\cite{CiMi:sep2}. In~\cite{CiMi:sep2}, Cicalese and Milani\v c ask whether for some fixed $k$, the class of graphs of separability at most $k$ is $\chi$-bounded, where a graph has \emph{separability at most $k$} if every two non-adjacent vertices are separated by a set of at most $k$ other vertices. Since graphs of separability at most $k$ form a subclass of the closure of the class of all complete graphs under gluing along at most $k$ vertices, Theorem \ref{small cutsets} implies that graphs of separability at most $k$ are $\chi$-bounded by the linear function $f(x) = x+2k^2-1$. We also note that the fact that the closure of a $\chi$-bounded class under gluing along at most $k$ vertices is again $\chi$-bounded also follows from an earlier result due to a group of authors \cite{alon}. However, the proof presented in this paper is significantly different from the one given in \cite{alon}, and furthermore, the $\chi$-bounding function that we obtained is better than the one that can be derived using the result from \cite{alon} (see section~\ref{section:gluing} for a more detailed explanation).
In section \ref{section:gluing}, we also show that the closure of a
$\chi$-bounded class under both of our gluing operations (gluing along 
a clique and gluing along at most $k$ vertices) together is 
$\chi$-bounded (see Proposition \ref{gluing combination}).  At the end of the 
section, we prove that that the closure of a $\chi$-bounded class 
under substitution and gluing along a clique together is $\chi$-bounded 
(see Proposition \ref{closure}, as well as Proposition \ref{gluing polyexp} for a strengthening 
of Proposition \ref{closure} in some special cases).
\\
\\
Finally, in section \ref{section:questions}, we state some open
questions related to $\chi$-boundedness.

\subsection{Definitions and Notation} \label{subsection:definitions} 

 The set of all non-negative integers is denoted by $\mathbb{N}$, and the set of all positive integers is denoted by $\mathbb{Z}^+$. Given a real number $r$, we denote by $\lfloor r \rfloor$ the largest integer that is no greater than $r$. Given a function $f:A \rightarrow B$ and a set $A' \subseteq A$, we denote by $f \upharpoonright A'$ the restriction of $f$ to $A'$, and we denote by $f[A']$ the set of all $b \in B$ such that for some $a \in A'$, $b = f(a)$. 
\\
\\
The vertex set of a graph $G$ is denoted by $V_G$. A non-empty graph is said to be {\em trivial} if it has only one vertex, and it is said to be {\em non-trivial} if it has at least two vertices; the empty graph is not referred to as either trivial or non-trivial. An {\em optimal} coloring of a graph $G$ is a proper coloring of $G$ that uses only $\chi(G)$ colors. A {\em triangle} in $G$ is a three-vertex clique in $G$. A {\em stable set} in $G$ is a set of pairwise non-adjacent vertices in $G$. An {\em isolated vertex} in $G$ is a vertex of $G$ that has no neighbors. Given a set $S \subseteq V_G$, we denote by $G[S]$ the subgraph of $G$ induced by $S$. Given a vertex $v \in V_G$ and a set $S \subseteq V_G \smallsetminus \{v\}$, we say that $v$ is {\em complete} (respectively: {\em anti-complete}) to the set $S$ or to the induced subgraph $G[S]$ of $G$ provided that $v$ is adjacent (respectively: non-adjacent) to every vertex in $S$; $v$ is said to be {\em mixed} on the set $S$ or the induced subgraph $G[S]$ if $v$ is neither complete nor anti-complete to $S$. A non-empty set $S \subseteq V_G$ is a {\em homogeneous set} in $G$ provided that no vertex in $V_G \smallsetminus S$ is mixed on $S$. (Thus, if $G$ is a non-empty graph, then $V_G$ is a homogeneous set in $G$, as is every one-vertex subset of $V_G$.) Given disjoint sets $A,B \subseteq V_G$, we say that $A$ or $G[A]$ is {\em complete} (respectively: {\em anti-complete}) to $B$ or $G[B]$ in $G$ provided that every vertex in $A$ is complete (respectively: anti-complete) to $B$. Given a graph $G$ and a set $S \subseteq V_G$, we denote by $G \smallsetminus S$ the graph obtained from $G$ by deleting all vertices in $S$; if $S = \{v\}$, we often write $G \smallsetminus v$ instead of $G \smallsetminus S$. 
\\
\\
Given non-empty graphs $G_1$ and $G_2$ with disjoint vertex sets, a vertex $u \in G_1$, and a graph $G$, we say that $G$ is obtained by {\em substituting} $G_2$ for $u$ in $G_1$ provided that the following hold: 
\begin{itemize} 
\item $V_G = (V_{G_1} \smallsetminus \{u\}) \cup V_{G_2}$; 
\item $G[V_{G_1} \smallsetminus \{u\}] = G_1 \smallsetminus u$; 
\item $G[V_{G_2}] = G_2$; 
\item for all $v \in V_{G_1} \smallsetminus \{u\}$, if $v$ is adjacent (respectively: non-adjacent) to $u$ in $G_1$, then $v$ is adjacent  (respectively: non-adjacent) to every vertex in $V_{G_2}$ in $G$. 
\end{itemize} 
In this paper, we will often need a slightly different substitution operation, one that allows us to substitute graphs for all vertices of the ``base graph'' simultaneously. More precisely, given a non-empty graph $G_0$ with vertex set $V_{G_0} = \{v_1,...,v_t\}$ and non-empty graphs $G_1,...,G_t$ with pairwise disjoint vertex sets, we say that a graph $G$ is obtained by {\em substituting} $G_1,...,G_t$ for $v_1,...,v_t$ in $G_0$ provided that the following hold: 
\begin{itemize} 
\item $V_G = \bigcup_{i=1}^t V_{G_i}$; 
\item for all $i \in \{1,...,t\}$, $G[V_{G_i}] = G_i$; 
\item for all distinct $i,j \in \{1,...,t\}$, if $v_i$ is adjacent (respectively: non-adjacent) to $v_j$ in $G_0$, then $V_{G_i}$ is complete (respectively: anti-complete) to $V_{G_j}$ in $G$. 
\end{itemize} 
Clearly, any graph that can be obtained by simultaneous substitution can also be obtained by sequentially applying ordinary substitution. Conversely, since the graphs used in simultaneous substitution may possibly be trivial, the operation of simultaneous substitution is no more restrictive than ordinary substitution. However, simultaneous substitution is often more convenient to work with. 
\\
\\
Next, we define a certain ``gluing operation'' as follows. Let $G_1$ and $G_2$ be non-empty graphs with inclusion-wise incomparable vertex sets, and let $C = V_{G_1} \cap V_{G_2}$. Assume that $G_1[C] = G_2[C]$. Let $G$ be a graph such that $V_G = V_{G_1} \cup V_{G_2}$, with adjacency as follows: 
\begin{itemize} 
\item $G[V_{G_1}] = G_1$; 
\item $G[V_{G_2}] = G_2$; 
\item $V_{G_1} \smallsetminus C$ is anti-complete to $V_{G_2} \smallsetminus C$ in $G$. 
\end{itemize} 
We then say that $G$ is obtained by {\em gluing $G_1$ and $G_2$ along $C$}. Under these circumstances, we also say that $G$ is obtained by gluing $G_1$ and $G_2$ along $|C|$ vertices. If $C$ is, in addition, a (possibly empty) clique in both $G_1$ and $G_2$, then we say that $G$ is obtained from $G_1$ and $G_2$ by {\em gluing along a clique}. We observe that gluing two graphs with disjoint vertex sets along the empty set (equivalently: along the empty clique) simply amounts to taking the disjoint union of the two graphs; thus, if a hereditary class $\mathcal{G}$ is closed under gluing along a clique, then $\mathcal{G}$ is also closed under taking disjoint unions. 
\\
\\
Given a positive integer $k$ and a class $\mathcal{G}$ of graphs, we say that $\mathcal{G}$ is {\em closed under gluing along at most $k$ vertices} provided that for all non-empty graphs $G_1,G_2 \in \mathcal{G}$ with inclusion-wise incomparable vertex sets, if $G_1[V_{G_1} \cap V_{G_2}] = G_2[V_{G_1} \cap V_{G_2}]$ and $|V_{G_1} \cap V_{G_2}| \leq k$, then the graph obtained by gluing $G_1$ and $G_2$ along $V_{G_1} \cap V_{G_2}$ is a member of $\mathcal{G}$. 
\\
\\
We observe that (like substitution) the operation of gluing along a clique preserves hereditariness, as does the operation of gluing along a bounded number of vertices.

\section{Substitution} \label{section:substitution} 

Given a class $\mathcal{G}$ of graphs, we denote by $\mathcal{G}^+$ the closure of $\mathcal{G}$ under taking disjoint unions, and we denote by $\mathcal{G}^*$ the closure of $\mathcal{G}$ under taking disjoint unions and substitution. In this section, we show that if $\mathcal{G}$ is a $\chi$-bounded class, then the class $\mathcal{G}^*$ is also $\chi$-bounded (see Theorem \ref{general bound}). We then improve on this result in a number of special cases: when the $\chi$-bounding function for $\mathcal{G}$ is polynomial (see Theorem \ref{poly}), when it is supermultiplicative (see Proposition \ref{supermult}), and when it is exponential (see Theorem \ref{exp}).

\subsection{Substitution depth and $\chi$-boundedness}

Let $\mathcal{G}$ be a hereditary class. We note that if $\mathcal{G}$ contains even one non-empty graph, then $\mathcal{G}^+$ contains all the edgeless graphs; we also note that if $\mathcal{G}$ is $\chi$-bounded by a non-decreasing function $f$, then $\mathcal{G}^+$ is also $\chi$-bounded by $f$. We observe that every graph $G \in \mathcal{G}^* \smallsetminus \mathcal{G}^+$ can be obtained from a graph $G_0 \in \mathcal{G}^+$ with vertex set $V_{G_0} = \{v_1,...,v_t\}$ (where $2 \leq t \leq |V_G|-1$) and non-empty graphs $G_1,...,G_t \in \mathcal{G}^*$ with pairwise disjoint vertex sets by substituting $G_1,...,G_t$ for $v_1,...,v_t$ in $G_0$. We now define the {\em substitution depth} of the graphs $G \in \mathcal{G}^*$ with respect to $\mathcal{G}$, denoted by $d_{\mathcal{G}}(G)$, as follows. If $G$ is the empty graph, then set $d_{\mathcal{G}}(G) = -1$. For all non-empty graphs $G \in \mathcal{G}^+$, set $d_{\mathcal{G}}(G) = 0$. Next, let $G \in \mathcal{G}^* \smallsetminus \mathcal{G}^+$, and assume that $d_{\mathcal{G}}(G')$ has been defined for every graph $G' \in \mathcal{G}^* \smallsetminus \mathcal{G}^+$ with at most $|V_G|-1$ vertices. Then we define $d_{\mathcal{G}}(G)$ to be the smallest non-negative integer $r$ such that there exist non-empty graphs $G_1,...,G_t \in \mathcal{G}^*$ (where $2 \leq t \leq |V_G|-1$) with pairwise disjoint vertex sets, and a graph $G_0 \in \mathcal{G}^+$ with vertex set $V_{G_0} = \{v_1,...,v_t\}$, where $v_1,...,v_s$ (for some $s \in \{0,...,t\}$) are isolated vertices in $G_0$ and each of $v_{s+1},...,v_t$ has a neighbor in $G_0$, such that $G$ is obtained by substituting $G_1,...,G_t$ for $v_1,...,v_t$ in $G_0$, and 
\begin{displaymath} 
r = \max(\{d_{\mathcal{G}}(G_1),...,d_{\mathcal{G}}(G_s)\} \cup \{d_{\mathcal{G}}(G_{s+1})+1,...,d_{\mathcal{G}}(G_t)+1\}). 
\end{displaymath} 
We observe that the fact that $\mathcal{G}$ is hereditary implies that $d_{\mathcal{G}}(H) \leq d_{\mathcal{G}}(G)$ for all graphs $G \in \mathcal{G}^*$, and all induced subgraphs $H$ of $G$. We now prove a technical lemma. 
\begin{lemma} \label{depth} Let $\mathcal{G}$ be a hereditary class, $\chi$-bounded by a non-decreasing function $f:\mathbb{N} \rightarrow \mathbb{R}$. Then for all $G \in \mathcal{G}^*$, we have that $\omega(G) \geq d_{\mathcal{G}}(G)+1$ and $\chi(G) \leq f(\omega(G))^{d_{\mathcal{G}}(G)+1}$. 
\end{lemma} 
\begin{proof} 
We proceed by induction on the number of vertices. Fix $G \in \mathcal{G}^*$, and assume that the claim holds for graphs in $\mathcal{G}^*$ that have fewer vertices than $G$. If $G \in \mathcal{G}^+$, then the result is immediate, so assume that $G \notin \mathcal{G}^+$. Fix $G_0 \in \mathcal{G}^+$ with vertex set $V_{G_0} = \{v_1,...,v_t\}$ (with $2 \leq t \leq |V_G|-1$), where $v_1,...,v_s$ (with $s \in \{0,...,t\}$) are isolated vertices in $G_0$ and each of $v_{s+1},...,v_t$ has a neighbor in $G_0$, and non-empty graphs $G_1,...,G_t \in \mathcal{G}^*$ such that $G$ is obtained by substituting $G_1,...,G_t$ for $v_1,...,v_t$ in $G_0$, and $d_{\mathcal{G}}(G) = \max(\{d_{\mathcal{G}}(G_1),...,d_{\mathcal{G}}(G_s)\} \cup \{d_{\mathcal{G}}(G_{s+1})+1,...,d_{\mathcal{G}}(G_t)+1\})$. 
\\
\\
We first show that $\omega(G) \geq d_{\mathcal{G}}(G)+1$. We need to show that $\omega(G) \geq d_{\mathcal{G}}(G_i)+1$ for all $i \in \{1,...,s\}$, and that $\omega(G) \geq d_{\mathcal{G}}(G_i)+2$ for all $i \in \{s+1,...,t\}$. By the induction hypothesis, we have that $\omega(G_i) \geq d_{\mathcal{G}}(G_i)+1$ for all $i \in \{1,...,t\}$, and so it suffices to show that $\omega(G) \geq \omega(G_i)$ for all $i \in \{1,...,s\}$, and that $\omega(G) \geq \omega(G_i)+1$ for all $i \in \{s+1,...,t\}$. The former follows from the fact that $G_i$ is an induced subgraph of $G$ for all $i \in \{1,...,s\}$. For the latter, fix $i \in \{s+1,...,t\}$, and let $K$ be a clique of size $\omega(G_i)$ in $G_i$. Let $v_j$ be a neighbor of $v_i$ in $G_0$. Now fix $k \in V_{G_j}$, and note that $K \cup \{k\}$ is a clique of size $\omega(G_i)+1$ in $G$. 
\\
\\
It remains to show that $\chi(G) \leq f(\omega(G))^{d_{\mathcal{G}}(G)+1}$. Since $\chi(H)$ is non-negative integer for every graph $H$, we know that the class $\mathcal{G}$ is $\chi$-bounded by the function given by $n \mapsto \lfloor f(n) \rfloor$; thus, we may assume without loss of generality that $f(n)$ is a non-negative integer for all $n \in \mathbb{N}$. Note that $V_{G_i}$ is anti-complete to $V_G \smallsetminus V_{G_i}$ for all $i \in \{1,...,s\}$. Thus, it suffices to show that $\chi(G_i) \leq f(\omega(G))^{d_{\mathcal{G}}(G)+1}$ for all $i \in \{1,...,s\}$, and that $\chi(G[\bigcup_{i=s+1}^t V_{G_i}]) \leq f(\omega(G))^{d_{\mathcal{G}}(G)+1}$. The former is immediate from the induction hypothesis. For the latter, we use the induction hypothesis to assign a coloring $b_i:V_{G_i} \rightarrow \{1,...,f(\omega(G))^{d_{\mathcal{G}}(G)}\}$ to $G_i$ for each $i \in \{s+1,...,t\}$. Next, we use the fact that $G_0 \in \mathcal{G}^+$ and that $\mathcal{G}$ (and therefore $\mathcal{G}^+$ as well) is $\chi$-bounded by $f$ in order to assign a coloring $b_0:V_{G_0} \rightarrow \{1,...,f(\omega(G))\}$ to $G_0$. Now define $b:V_{G[\bigcup_{i=s+1}^t V_{G_i}]} \rightarrow \{1,...,f(\omega(G))\} \times \{1,...,f(\omega(G))^{d_{\mathcal{G}}(G)}\}$ by setting $b(v) = (b_0(v_i),b_i(v))$ for all $i \in \{s+1,...,t\}$ and $v \in V_{G_i}$. This is clearly a proper coloring of $G[\bigcup_{i=s+1}^t V_{G_i}]$ that uses at most $f(\omega(G))^{d_{\mathcal{G}}(G)+1}$ colors. 
\end{proof} 
\noindent 
As an immediate corollary, we have the following theorem. 
\begin{theorem} \label{general bound} Let $\mathcal{G}$ be a class of graphs, $\chi$-bounded by a non-decreasing function $f:\mathbb{N} \rightarrow \mathbb{R}$. Then the class $\mathcal{G}^*$ is $\chi$-bounded by the function $g(k) = f(k)^k$. 
\end{theorem} 
\begin{proof} 
We may assume that $\mathcal{G}$ is hereditary, because otherwise, instead of considering $\mathcal{G}$, we consider the closure $\tilde{\mathcal{G}}$ of $\mathcal{G}$ under isomorphism and taking induced subgraphs. (We may do this because $\tilde{\mathcal{G}}$ is readily seen to be hereditary and $\chi$-bounded by $f$, and furthermore, $\mathcal{G}^* \subseteq \tilde{\mathcal{G}}^*$, and so if $\tilde{\mathcal{G}}^*$ is $\chi$-bounded by $g$, then so is $\mathcal{G}^*$.) 
\\
\\
We may assume that $f(0) \geq 0$ and that $f(k) \geq 1$ for all $k \in \mathbb{Z}^+$, for otherwise, $\mathcal{G}$ contains no non-empty graphs, and the result is immediate. Next, if $H$ is the empty graph, then $\chi(H) = 0 \leq 1 = f(\omega(H))^{\omega(H)}$. Finally, suppose that $G \in \mathcal{G}^*$ is a non-empty graph. Now, by Lemma \ref{depth}, we have that $\chi(G) \leq f(\omega(G))^{d_{\mathcal{G}}(G)+1}$ and $d_{\mathcal{G}}(G)+1 \leq \omega(G)$; since $f(\omega(G)) \geq 1$, it follows that $\chi(G) \leq f(\omega(G))^{\omega(G)}$. 
\end{proof} 

\subsection{Polynomial $\chi$-bounding functions}

We now turn to the special case when a hereditary class $\mathcal{G}$ is $\chi$-bounded by a polynomial function. 
\begin{theorem} \label{poly} Let $\mathcal{G}$ be a $\chi$-bounded class. If $\mathcal{G}$ has a polynomial $\chi$-bounding function, then so does $\mathcal{G}^*$. 
\end{theorem} 
\begin{proof} 
We may assume that $\mathcal{G}$ is hereditary (otherwise, instead of $\mathcal{G}$, we consider the closure of $\mathcal{G}$ under isomorphism and taking induced subgraphs). Further, we may assume that $\mathcal{G}$ is $\chi$-bounded by the function $f(x) = x^A$ for some $A \in \mathbb{Z}^+$. Set $g(x) = x^{3A+11}$, and set $B = 2A+11$, so that $g(x) = x^{A+B}$. Our goal is to show that $\mathcal{G}^*$ is $\chi$-bounded by the function $g$. Fix a graph $G \in \mathcal{G}^*$, set $d_{\mathcal{G}}(G) = t$, and assume inductively that for every graph $G' \in \mathcal{G}^*$ with $d_{\mathcal{G}}(G') < t$, we have that $\chi(G') \leq g(\omega(G'))$. Set $\omega = \omega(G)$. We need to show that $\chi(G) \leq g(\omega)$. If $G$ is the empty graph, then the result is immediate; so we may assume that $G$ is a non-empty graph. 
\\
\\
By Lemma \ref{depth}, if $t \leq 2$, then $\chi(G) \leq f(\omega(G))^3 < g(\omega(G))$, and we are done. So from now on, we assume that $t \geq 3$. Lemma \ref{depth} then implies that $\omega \geq 4$. Next, since $d_{\mathcal{G}}(H) \leq d_{\mathcal{G}}(G)$ for every induced subgraph $H$ of $G$, we may assume that $G$ is connected (for otherwise, we deal with the components of $G$ separately). Thus, there exists a connected graph $F \in \mathcal{G}$ with vertex set $V_F = \{v_1,...,v_n\}$ (with $n \geq 2$), and non-empty graphs $B_1,...,B_n \in \mathcal{G}^*$, with $d_{\mathcal{G}}(B_i) < t$ for all $i \in \{1,...,n\}$, such that $G$ is obtained by substituting $B_1,...,B_n$ for $v_1,...,v_n$ in $F$. For all $i \in \{1,...,n\}$, set $\omega_i = \omega(B_i)$. Note that by the induction hypothesis, we have that $\chi(B_i) \leq g(\omega_i)$ for all $i \in \{1,...,n\}$. We observe that if $v_i,v_j \in V_F$ are adjacent, then $\omega_i+\omega_j \leq \omega$; since $F$ contains no isolated vertices, it follows that $\omega_i \leq \omega-1$ for all $i \in \{1,...,n\}$. 
\\
\\
Fix $\alpha \in [\frac{5}{4},\frac{3}{2}]$ such that $\alpha^m = \frac{\omega}{2}$ for some $m \in \mathbb{Z}^+$; such an $\alpha$ exists because $\{\hat{\alpha}^k \mid k \in \mathbb{Z}^+, \hat{\alpha} \in [\frac{5}{4},\frac{3}{2}]\} = [\frac{5}{4},\frac{3}{2}] \cup [\frac{25}{16},+\infty)$ and $\frac{\omega}{2} \geq 2$. We now define: 
\begin{displaymath} 
\begin{array}{lcll} 
V_0 & = & \{v_i \mid \omega_i > \frac{\omega}{2}\}, & 
\\
V_j & = & \{v_i \mid \omega_i \in (\frac{\omega}{2\alpha^j},\frac{\omega}{2\alpha^{j-1}}]\}, & 1 \leq j \leq m, 
\\
V_{m+1} & = & \{v_i \mid \omega_i = 1\}, 
\end{array} 
\end{displaymath} 
so that the sets $V_0,V_1,...,V_{m+1}$ are pairwise disjoint with $V_F = \bigcup_{j=0}^{m+1} V_j$. For each $j \in \{0,...,m+1\}$, set $F_j = F[V_j]$, and let $G_j$ be the corresponding induced subgraph of $G$ (formally: $G_j = G[\bigcup_{v_i \in V_j} V_{B_i}]$). 
\\
\\
Note that if $C$ is a clique in $F$, then 
\begin{equation} \label{stab1} 
\omega \geq \Sigma_{v_i \in C \phantom{i}} \omega_i. 
\end{equation} 
In particular, $V_0$ is a stable set. Further, for all $j \in \{1,...,m\}$, if $v_i \in V_j$ then $\omega_i \geq \frac{\omega}{2\alpha^j}$; by (\ref{stab1}), this implies that $\omega \geq \omega(F_j) \cdot \frac{\omega}{2\alpha^j}$, and so $\omega(F_j) \leq 2\alpha^j$. But now for each $j \in \{1,...,m\}$, we have: 
\begin{equation} \label{fg1} 
\begin{array}{rcl} 
\chi(G_j) & \leq & \chi(F_j) \cdot \max_{v_i \in V_j}\chi(B_i) 
\\
& \leq & \chi(F_j) \cdot \max_{v_i \in V_j} g(\omega_i) 
\\
& \leq & f(2\alpha^j)g(\frac{\omega}{2\alpha^{j-1}}). 
\end{array} 
\end{equation} 
We also have that: 
\begin{equation} \label{fg2} 
\chi(G_{m+1}) = \chi(F_{m+1}) \leq f(\omega). 
\end{equation} 
We now color $G$ as follows: 
\begin{itemize} 
\item we first color each subgraph $G_j$, $j \in \{1,...,m+1\}$, with a separate set of colors (using in each case only $\chi(G_j)$ colors); 
\item we then color the subgraphs $B_i$ with $v_i \in V_0$ one at a time, introducing at each step as few new colors as possible. 
\end{itemize} 
We need to show that this coloring of $G$ uses at most $g(\omega)$ colors. 
\\
\\
From (\ref{fg1}) and (\ref{fg2}), we get that coloring the graphs $G_1,...,G_{m+1}$ together takes at most the following number of colors: 
\begin{equation} \label{fh1} 
\begin{array}{rcl} 
\Sigma_{j=1}^{m+1} \chi(G_j) & \leq & f(\omega)+\Sigma_{j=1}^m f(2\alpha^j)g(\frac{\omega}{2\alpha^{j-1}}) 
\\
& = & \omega^A+\Sigma_{j=1}^m (2\alpha^j)^A(\frac{\omega}{2\alpha^{j-1}})^{3A+11} 
\\
& = & \omega^A+(\alpha\omega)^A \Sigma_{j=1}^m (\frac{\omega}{2\alpha^{j-1}})^B 
\\
& = & \omega^{A+B}(\omega^{-B}+\frac{\alpha^A}{2^B} \Sigma_{j=0}^{m-1} (\alpha^{-B})^j) 
\\
& \leq & \omega^{A+B}(\omega^{-B}+\frac{\alpha^A}{2^B} \frac{1}{1-\alpha^{-B}}) 
\\ 
& = & g(\omega)(\omega^{-B}+\frac{\alpha^A}{2^B} \frac{1}{1-\alpha^{-B}}) 
\\
& \leq & g(\omega)(\frac{1}{2^B}+\frac{(\frac{3}{2})^A}{2^B}\frac{1}{1-(\frac{5}{4})^{-B}}) 
\\
& \leq & g(\omega)(\frac{1}{2^B}+\frac{(\frac{3}{2})^A}{2^B}\frac{1}{1-\frac{4}{5}}) 
\\
& = & g(\omega) \cdot \frac{1+5(\frac{3}{2})^A}{2^B} 
\\
& \leq & g(\omega) \cdot \frac{6(\frac{3}{2})^A}{2^{2A+11}} 
\\
& \leq & g(\omega). 
\end{array} 
\end{equation} 
Now consider the graphs $B_i$ with $v_i \in V_0$. These are pairwise anti-complete to each other (as $V_0$ is stable). Fix $v_i \in V_0$. It suffices to show that our coloring of $G$ used no more than $g(\omega)$ colors on $B_i$ and all the vertices with a neighbor in $B_i$. Note that if a vertex $v_j$ is adjacent to $v_i$ in $F$, then $V_{B_j}$ is complete to $V_{B_i}$ in $G$, and so $\omega_i+\omega_j \leq \omega$; thus, all neighbors of $v_i$ lie in 
\begin{displaymath} 
V_{m+1} \cup \{V_j \mid 1 \leq j \leq m, \frac{\omega}{2\alpha^j} < \omega-\omega_i\}. 
\end{displaymath} 
Let $s_i = min\{s \in \mathbb{Z}^+ \mid \frac{\omega}{2\alpha^s} < \omega-\omega_i\}$; $s_i$ is well-defined because $\omega_i < \omega$. Then using (\ref{fg1}) and (\ref{fg2}), we get that the number of colors already used in subgraphs $G_j$ that are not anti-complete to $B_i$ is at most: 
\begin{equation} 
\begin{array}{rcl} \label{fh2} 
\chi(G_{m+1})+\Sigma_{j=s_i}^m \chi(G_j) & \leq & f(\omega)+\Sigma_{j=s_i}^m f(2\alpha^j)g(\frac{\omega}{2\alpha^{j-1}}) 
\\
& = & f(\omega)+\Sigma_{j=s_i}^m (2\alpha^j)^A(\frac{\omega}{2\alpha^{j-1}})^{3A+11} 
\\
& = & f(\omega)+(\alpha\omega)^A\Sigma_{j=s_i}^m (\frac{\omega}{2\alpha^{j-1}})^B
\\
& = & f(\omega)+f(\alpha\omega)\Sigma_{j=0}^{m-s_i} (\frac{\omega}{2\alpha^{s_i+j-1}})^B 
\\
& = & f(\omega)+f(\alpha\omega)\Sigma_{j=0}^{m-s_i} (\frac{\alpha}{\alpha^j} \cdot \frac{\omega}{2\alpha^{s_i}})^B 
\\
& \leq & f(\omega)+f(\alpha\omega)\Sigma_{j=0}^{m-s_i} (\frac{\alpha(\omega-\omega_i)}{\alpha^j})^B. 
\end{array} 
\end{equation} 
Set $p = 1-\frac{\omega_i}{\omega}$; note that we then have that $p \in [\frac{1}{\omega},\frac{1}{2})$, as $\frac{\omega}{2} < \omega_i \leq \omega-1$. Now, we use at most $g(\omega_i) = g((1-p)\omega)$ colors on $B_i$, which together with (\ref{fh2}) implies that we use at most 
\begin{equation} 
P = f(\omega)+f(\alpha\omega)\Sigma_{j=0}^{m-s_i} (\frac{\alpha p\omega}{\alpha^j})^B+g((1-p)\omega) 
\end{equation} 
colors on $B_i$ and all the $G_j$ that are not anti-complete to $B_i$ together; our goal is to show that $P \leq g(\omega)$. Note the following: 
\begin{displaymath} 
\begin{array}{rcl} 
P & = & f(\omega)+f(\alpha\omega)\Sigma_{j=0}^{m-s_i} (\frac{\alpha p\omega}{\alpha^j})^B+g((1-p)\omega) 
\\
\\
& = & \omega^A+\alpha^A\omega^A \Sigma_{j=0}^{m-s_i} \frac{\alpha^Bp^B\omega^B}{\alpha^{jB}}+(1-p)^{A+B}\omega^{A+B} 
\\
\\
& = & \omega^{A+B}(\omega^{-B}+\alpha^{A+B}p^B \Sigma_{j=0}^{m-s_i} \frac{1}{(\alpha^B)^j}+(1-p)^{A+B}) 
\\
\\
& \leq & \omega^{A+B}(\omega^{-B}+\alpha^{A+B}p^B \Sigma_{j=0}^\infty \frac{1}{(\alpha^B)^j}+(1-p)^{A+B}) 
\\
\\
& = & g(\omega)(\omega^{-B}+\frac{\alpha^{A+B}p^B}{1-\alpha^{-B}}+(1-p)^{A+B}) 
\\
\\
& \leq & g(\omega)(2\frac{\alpha^{A+B}p^B}{1-\alpha^{-B}}+(1-p)^{A+B}). 
\end{array} 
\end{displaymath} 
(In the last step, we used the fact that $\frac{\alpha^{A+B}}{1-\alpha^{-B}} \geq 1$ and $p \geq \frac{1}{\omega}$.) Thus, in order to show that $P \leq g(\omega)$, it suffices to show that $2\frac{\alpha^{A+B}p^B}{1-\alpha^{-B}}+(1-p)^{A+B} \leq 1$. First, using the fact that $\frac{5}{4} \leq \alpha \leq \frac{3}{2}$ and $0 \leq p \leq \frac{1}{2}$ (and consequently, $\alpha p \leq \frac{3}{4}$), we get that: 
\begin{displaymath} 
\begin{array}{rcl} 
2\frac{\alpha^{A+B}p^B}{1-\alpha^{-B}} & = & 2\alpha^A\frac{(\alpha p)^B}{1-\alpha^{-B}} 
\\
& \leq & 2(\frac{3}{2})^A\frac{(\frac{3}{4})^B}{1-\frac{4}{5}} 
\\
& = & 10(\frac{3}{2})^A(\frac{3}{4})^{2A+11} 
\\
& = & 10(\frac{27}{32})^A(\frac{3}{4})^{11}
\\
& \leq & 10 \cdot (\frac{3}{4})^{11} 
\\
& \leq & \frac{1}{2}. 
\end{array} 
\end{displaymath} 
On the other hand, we have that $(1-p)^{A+B} \leq e^{-p(A+B)}$, and so if $p \geq \frac{1}{A+B}$, then 
\begin{displaymath} 
\begin{array}{rcl} 
2\frac{\alpha^{A+B}p^B}{1-\alpha^{-B}}+(1-p)^{A+B} & \leq & \frac{1}{2}+\frac{1}{e} < 1, 
\end{array} 
\end{displaymath} 
and we are done. So assume that $p < \frac{1}{A+B}$. Note first that: 
\begin{displaymath} 
\begin{array}{rcl} 
\frac{2\alpha^{A+B}}{1-\alpha^{-B}} & \leq & \frac{2(\frac{3}{2})^{A+B}}{1-\frac{4}{5}} 
\\
& = & 10(\frac{3}{2})^{3A+11} 
\\
& = & 10(\frac{3}{2})^{11}(\frac{27}{8})^A 
\\
& \leq & 4^{11} \cdot 4^A 
\\
& \leq & 4^B. 
\end{array} 
\end{displaymath} 
Now, since $p < \frac{1}{A+B}$, we have that $4p \leq 1$ and $p(A+B) \leq 1$, and consequently, that $(4p)^B \leq 4p$ and $(p(A+B))^2 \leq p(A+B)$. But now we have the following: 
\begin{displaymath} 
\begin{array}{rcl} 
2\frac{\alpha^{A+B}p^B}{1-\alpha^{-B}}+(1-p)^{A+B} & \leq & 4^Bp^B+e^{-p(A+B)} 
\\
& \leq & (4p)^B+(1-p(A+B)+\frac{(p(A+B))^2}{2}) 
\\
& \leq & 4p+(1-p(A+B)+\frac{p(A+B)}{2}) 
\\
& = & 1-(\frac{A+B-8}{2})p 
\\
& = & 1-\frac{3A+3}{2}p 
\\
& < & 1. 
\end{array} 
\end{displaymath} 
This completes the argument. 
\end{proof} 
\noindent 
It is natural to ask whether Theorem \ref{poly} could be improved by bounding the degree of $g$ in terms of the degree of $f$. However, the following proposition (Proposition \ref{fractional}) shows that this is not possible. In what follows, $\chi_f(G)$ denotes the fractional chromatic number of the graph $G$. The proof of the proposition uses the fact that there exist triangle-free graphs of arbitrarily large fractional chromatic number; this follows immediately from the fact that the Ramsey number $R(3,t)$ satisfies $\frac{R(3,t)}{t}\to\infty$, which follows from standard probabilistic arguments (in fact, $R(3,t)$ has order of magnitude $\frac{t^2}{\log t}$; the upper bound was established in \cite{ajtai} and the lower bound in \cite{Kim}). 
\begin{prop} \label{fractional} For every $d \in \mathbb{Z}^+$, there is a hereditary class $\mathcal{G}$, $\chi$-bounded by a linear $\chi$-bounding function, such that every polynomial $\chi$-bounding function of $\mathcal{G}^*$ has degree greater than $d$. 
\end{prop} 
\begin{proof} 
Fix $d \in \mathbb{Z}^+$. Let $F$ be a graph with $\omega(F) = 2$ and $\chi_f(F) > 2^d$. Let $\mathcal{G}$ be the class that consists of all the isomorphic copies of $F$ and its induced subgraphs, as well as all the complete graphs. Then $\mathcal{G}$ is a hereditary class, $\chi$-bounded by the linear function $f(x) = x+\chi(F)$. Suppose that $\mathcal{G}^*$ is $\chi$-bounded by a polynomial function $g$ of degree at most $d$; we may assume that $g(x) = Mx^d$ for some $M \in \mathbb{Z}^+$. 
\\
\\
Define a sequence $F_1,F_2,...$ as follows. Set $F_1 = F$, and for each $i \in \mathbb{Z}^+$, let $F_{i+1}$ be the graph with vertex set $V_F \times V_{F_i}$ in which vertices $(u_1,v_1),(u_2,v_2) \in V_{F_{i+1}}$ are adjacent if and only if either $u_1$ and $u_2$ are adjacent in $F$, or $u_1 = u_2$ and $v_1$ and $v_2$ are adjacent in $V_{F_i}$; note that this means that $F_{i+1}$ is obtained by substituting a copy $F_i^v$ of $F_i$ for every vertex $v$ of $F$, and so $F_i \in \mathcal{G}^*$ for all $i \in \mathbb{Z}^+$. For each $i \in \mathbb{Z}^+$, let $\mathcal{S}_i$ be the set of all stable sets in $F_i$, and set $\mathcal{S} = \mathcal{S}_1$. 
\\
\\
First, we note that it follows by an easy induction that $\omega(F_i) = \omega(F)^i = 2^i$ for all $i \in \mathbb{Z}^+$. Next, we argue inductively that $\chi_f(F_i) = \chi_f(F)^i$ for all $i \in \mathbb{Z}^+$. For $i = 1$, this is immediate. Now assume that $\chi_f(F_i) = \chi_f(F)^i$; we claim that $\chi_f(F_{i+1}) = \chi_f(F)^{i+1}$. 
\\
\\
We begin by showing that $\chi_f(F_{i+1}) \geq \chi_f(F)^{i+1}$. Let $(S,\lambda_S)_{S \in \mathcal{S}_{i+1}}$ be a fractional coloring of $F_{i+1}$ (where each stable set $S$ is taken with weight $\lambda_S \geq 0$) with $\Sigma_{S \in \mathcal{S}_{i+1}} \lambda_S = \chi_f(F_{i+1})$. For each $X \subseteq V_{F_{i+1}}$, set $\widehat{X} = \{u \in V_F \mid (u,v) \in X$ for some $v \in V_{F_i}\}$. Clearly, for all $S \in \mathcal{S}_{i+1}$, we have that $\widehat{S} \in \mathcal{S}$. For all $S' \in \mathcal{S}$, let $[S']_{i+1} = \{S \in \mathcal{S}_{i+1} \mid \widehat{S} = S'\}$; note that the set $\mathcal{S}_{i+1}$ is the disjoint union of the sets $[S']_{i+1}$ with $S' \in \mathcal{S}$. For each $S' \in \mathcal{S}$, set 
\begin{displaymath} 
\lambda_{S'} = \frac{\Sigma_{S \in [S']_{i+1}} \lambda_S}{\chi_f(F_i)}. 
\end{displaymath} 
Now, given $u \in V_F$, set $\mathcal{S}[u] = \{S \in \mathcal{S} \mid u \in S\}$ and $\mathcal{S}_{i+1}[u] = \{S \in \mathcal{S}_{i+1} \mid u \in \widehat{S}\}$, and note that for all $u \in V_F$, we have the following: 
\begin{displaymath} 
\begin{array}{rcl} 
\Sigma_{S' \in \mathcal{S}[u]} \lambda_{S'} & = & \Sigma_{S' \in \mathcal{S}[u]} \frac{\Sigma_{S \in [S']_{i+1}} \lambda_S}{\chi_f(F_i)} 
\\
\\
& = & \frac{1}{\chi_f(F_i)} \Sigma_{S \in \mathcal{S}_{i+1}[u]} \lambda_S
\\
\\
& \geq & \frac{\chi_f(F_i)}{\chi_f(F_i)} 
\\
\\
& = & 1. 
\end{array} 
\end{displaymath} 
Thus, $(S',\lambda_{S'})_{S \in \mathcal{S}}$ is a fractional coloring of $F$, and so $\Sigma_{S' \in \mathcal{S}} \lambda_{S'} \geq \chi_f(F)$. But now we have that: 
\begin{displaymath} 
\begin{array}{rcl} 
\chi_f(F) & \leq & \Sigma_{S' \in \mathcal{S}} \lambda_{S'} 
\\
\\
& = & \Sigma_{S' \in \mathcal{S}} \frac{\Sigma_{S \in [S']_{i+1}} \lambda_S}{\chi_f(F_i)} 
\\
\\
& = & \frac{\Sigma_{S \in \mathcal{S}_{i+1}} \lambda_S}{\chi_f(F_i)} 
\\
\\
& = & \frac{\chi_f(F_{i+1})}{\chi_f(F_i)}, 
\end{array} 
\end{displaymath} 
and so $\chi_f(F_{i+1}) \geq \chi_f(F)\chi_f(F_i) = \chi_f(F)^{i+1}$. 
\\
\\
It remains to construct a fractional coloring of $F_{i+1}$ in which the sum of weights is equal to $\chi_f(F)^{i+1}$, the lower bound for $\chi_f(F_{i+1})$ that we just obtained. First, let $(S,\lambda_S)_{S \in \mathcal{S}}$ be a fractional coloring of $F$ with $\Sigma_{S \in \mathcal{S}} \lambda_S = \chi_f(F)$, and let $(S,\lambda_S)_{S \in \mathcal{S}_i}$ be a fractional coloring of $F_i$ with $\Sigma_{S \in \mathcal{S}_i} \lambda_S = \chi_f(F_i)$. Next, for all $S \in \mathcal{S}_{i+1}$, if there exist some $S' \in \mathcal{S}$ and $S'' \in \mathcal{S}_i$ such that $S = S' \times S''$ then we set $\lambda_S = \lambda_{S'}\lambda_{S''}$, and otherwise we set $\lambda_S = 0$. But now $(S,\lambda_S)_{S \in \mathcal{S}_{i+1}}$ is a fractional coloring of $F_{i+1}$ with $\Sigma_{S \in \mathcal{S}_{i+1}} \lambda_S = \chi_f(F)\chi_f(F_i) = \chi_f(F)^{i+1}$. This completes the induction. 
\\
\\
Finally, from $\chi_f(F_i) \leq g(\omega(F_i))$, we get that $\chi_f(F)^i \leq M \cdot 2^{id}$ for all $i \in \mathbb{Z}^+$. But this implies that $\chi_f(F) \leq M^{1/i} \cdot 2^d$ for all $i \in \mathbb{Z}^+$, which is impossible since $\chi_f(F) > 2^d$ and $\lim_{i \rightarrow \infty} M^{1/i} = 1$. 
\end{proof} 

\subsection{Faster growing $\chi$-bounding functions}

A function $f:\mathbb{N} \rightarrow \mathbb{R}$ is said to be {\em supermultiplicative} provided that $f(m)f(n) \leq f(mn)$ for all $m,n \in \mathbb{Z}^+$. Our next result (Proposition \ref{supermult}) improves on Theorem \ref{general bound} in the case when the $\chi$-bounding function of a $\chi$-bounded class $\mathcal{G}$ is supermultiplicative. 
\begin{prop} \label{supermult} Let $\mathcal{G}$ a class of graphs, $\chi$-bounded by a supermultiplicative non-decreasing function $f:\mathbb{N} \rightarrow \mathbb{R}$. Then $\mathcal{G}^*$ is $\chi$-bounded by the function $g:\mathbb{N} \rightarrow \mathbb{R}$ given by $g(0) = 0$ and $g(x) = f(x)x^{\log_2x}$ for all $x \in \mathbb{Z}^+$. 
\end{prop} 
\begin{proof} 
We may assume that $\mathcal{G}$ is hereditary (otherwise, instead of $\mathcal{G}$, we consider the closure of $\mathcal{G}$ under isomorphism and induced subgraphs). Let $G \in \mathcal{G}^*$, set $d_{\mathcal{G}}(G) = t$, and assume inductively that $\chi(G') \leq g(\omega(G'))$ for all graphs $G' \in \mathcal{G}^*$ with $d_{\mathcal{G}}(G') < t$; we need to show that $\chi(G) \leq g(\omega(G))$. If $t = -1$, then $G$ is the empty graph, and the result is immediate. If $t = 0$, then $G$ is a non-empty graph in $\mathcal{G}^+$, and the result follows from the fact that $\mathcal{G}^+$ is $\chi$-bounded by $f$ and that $f(n) \leq g(n)$ for all $n \in \mathbb{Z}^+$. So assume that $t \geq 1$. By Lemma \ref{depth}, this means that $\omega(G) \geq 2$. We may assume that $G$ is connected, so that there exists a connected graph $F \in \mathcal{G}^+$ with vertex set $V_F = \{v_1,...,v_n\}$ (where $2 \leq n \leq |V_G|-1$), and non-empty graphs $B_1,...,B_n \in \mathcal{G}^*$ with pairwise disjoint vertex sets, and each with substitution depth at most $t-1$, such that $G$ is obtained by substituting $B_1,..,B_n$ for $v_1,...,v_n$ in $F$. Note that by the induction hypothesis, $\chi(B_i) \leq g(\omega(B_i))$ for all $i \in \{1,...,n\}$. 
\\
\\
Set $\omega = \omega(G)$, and for all $i \in \{1,...,n\}$, set $\omega_i = \omega(B_i)$. Next, for all $j \in \{1,...,\lfloor \frac{\omega}{2} \rfloor\}$, set $W_j = \{v_i \mid \omega_i = j\}$, and set $W_{\infty} = \{v_i \mid \omega_i > \frac{\omega}{2}\}$. For all $j \in \{1,...,\lfloor \frac{\omega}{2}\rfloor\}$, set $F_j = F[W_j]$ and $G_j = G[\bigcup_{v_i \in W_j} B_i]$, and set $F_{\infty} = F[W_{\infty}]$ and $G_{\infty} = G[\bigcup_{v_i \in W_{\infty}} B_i]$. Note that if $C$ is a clique in $F$, then we have that: 
\begin{equation} \label{stab1'} 
\omega \geq \Sigma_{v_i \in C \phantom{i}} \omega_i. 
\end{equation} 
Therefore, for all $j \in \{1,...,\lfloor \frac{\omega}{2}\rfloor\}$, we have that $\omega(F_j) \leq \lfloor \frac{\omega}{j} \rfloor$. But then for all $j \in \{1,...,\lfloor \frac{\omega}{2}\rfloor\}$, 
\begin{equation} \label{small omega} 
\begin{array}{rcl} 
\chi(G_j) & \leq & \chi(F_j) \cdot \max_{v_i \in W_j}\chi(B_i) 
\\
& \leq & f(\lfloor \frac{\omega}{j} \rfloor)g(j) 
\end{array} 
\end{equation} 
Furthermore, by (\ref{stab1'}) again, we have that $F_{\infty}$ is a stable set. Since $F$ contains no isolated vertices, we get by (\ref{stab1'}) that for all $i \in \{1,...,n\}$, $\omega_i \leq \omega-1$. Thus: 
\begin{equation} 
\begin{array}{rcl} \label{large omega} 
\chi(G_{\infty}) & \leq & \chi(F_{\infty}) \cdot \max_{v_i \in W_{\infty}}\chi(B_i) 
\\
& \leq & g(\omega-1) 
\end{array} 
\end{equation} 
But now using (\ref{small omega}) and (\ref{large omega}), we have the following: 
\begin{displaymath} 
\begin{array}{rcl} 
\chi(G) & \leq & \chi(G_{\infty})+\Sigma_{j=1}^{\lfloor \frac{\omega}{2} \rfloor} \chi(G_j) 
\\
& \leq & g(\omega-1)+\Sigma_{j=1}^{\lfloor \frac{\omega}{2} \rfloor} f(\lfloor \frac{\omega}{j} \rfloor)g(j) 
\\
& = & g(\omega-1)+\Sigma_{j=1}^{\lfloor \frac{\omega}{2} \rfloor} f(\lfloor \frac{\omega}{j} \rfloor)f(j)j^{\log_2j} 
\\
& \leq & g(\omega-1)+\Sigma_{j=1}^{\lfloor \frac{\omega}{2} \rfloor} f(\lfloor \frac{\omega}{j} \rfloor j)j^{\log_2j} 
\\
& \leq & g(\omega-1)+\Sigma_{j=1}^{\lfloor \frac{\omega}{2} \rfloor} f(\omega)j^{\log_2j} 
\\
& \leq & f(\omega)(\omega-1)^{\log_2\omega}+\frac{\omega}{2}f(\omega)(\frac{\omega}{2})^{\log_2(\frac{\omega}{2})} 
\\
& = & f(\omega)\omega^{\log_2\omega}(1-\frac{1}{\omega})^{\log_2\omega}+f(\omega)(\frac{\omega}{2})^{\log_2\omega} 
\\
& \leq & f(\omega)\omega^{\log_2\omega}(1-\frac{1}{\omega})+f(\omega)\frac{\omega^{\log_2\omega}}{\omega} 
\\
& = & f(\omega)\omega^{\log_2\omega} 
\\
& = & g(\omega) 
\end{array} 
\end{displaymath} 
This completes the argument. 
\end{proof} 
\noindent 
As a corollary of Proposition \ref{supermult}, we have the following result. 
\begin{theorem} \label{exp} Let $\mathcal{G}$ be a class of graphs, $\chi$-bounded by an exponential function. Then $\mathcal{G}^*$ is also $\chi$-bounded by an exponential function. 
\end{theorem} 
\begin{proof} 
We may assume that $\mathcal{G}$ is hereditary (otherwise, instead of $\mathcal{G}$, we consider the closure of $\mathcal{G}$ under isomorphism and induced subgraphs). We may assume that $\mathcal{G}$ is $\chi$-bounded by the function $f(x) = 2^{c(x-1)}$ for some $c \in \mathbb{Z}^+$. Then $f$ is a supermultiplicative non-decreasing function, and so by Proposition \ref{supermult}, $\mathcal{G}^*$ is $\chi$-bounded by the function $g:\mathbb{N} \rightarrow \mathbb{R}$ given by $g(0) = 0$ and $g(x) = f(x)x^{\log_2x}$ for all $x \in \mathbb{Z}^+$. But now note that for all $x \in \mathbb{Z}^+$, we have the following: 
\begin{displaymath} 
\begin{array}{rcl} 
g(x) & = & x^{\log_2x}f(x) 
\\
& = & 2^{(\log_2x)^2}2^{c(x-1)} 
\\
& \leq & 2^x2^{cx} 
\\
& = & 2^{(c+1)x} 
\end{array} 
\end{displaymath} 
Thus, $\mathcal{G}^*$ is $\chi$-bounded by the exponential function $h(x) = 2^{(c+1)x}$. 
\end{proof}

\section{Small cutsets, substitution and cliques} \label{section:gluing} 

In section \ref{section:substitution}, we saw that the closure of a $\chi$-bounded class under substitution is $\chi$-bounded. In this section, we prove analogous results for the operations of gluing along a clique (see Proposition \ref{gluing clique only}) and gluing along a bounded number of vertices (see Theorem \ref{small cutsets}). We then consider ``combinations'' of the three operations discussed in this paper, namely substitution, gluing along a clique, and gluing along a bounded number of vertices. In particular, we prove that the closure of a $\chi$-bounded class under gluing along a clique and gluing along a bounded number of vertices is $\chi$-bounded (see Proposition \ref{gluing combination}), as well as that the closure of a $\chi$-bounded class under gluing along a clique and substitution is $\chi$-bounded (see Proposition \ref{closure}). 

\subsection{Gluing Operations} 

We begin by giving an easy proof of the fact that the closure of a $\chi$-bounded class under gluing along a clique is again $\chi$-bounded. 

\begin{prop} \label{gluing clique only} Let $\mathcal{G}$ be a class of graphs, $\chi$-bounded by a non-decreasing function $f:\mathbb{N} \rightarrow \mathbb{R}$. Then the closure of $\mathcal{G}$ under gluing along a clique is also $\chi$-bounded by $f$. 
\end{prop} 
\begin{proof} 
Note that if a graph $G$ is obtained by gluing graphs $G_1$ and $G_2$ along a clique, then $\omega(G) = \max\{\omega(G_1),\omega(G_2)\}$ and $\chi(G) = \max\{\chi(G_1),\chi(G_2)\}$. The result now follows by an easy induction. 
\end{proof} 
\noindent 
We now turn to the question of gluing along a bounded number of vertices. Given a class $\mathcal{G}$ of graphs, and a positive integer $k$, let $\mathcal{G}^k$ denote the closure of $\mathcal{G}$ under gluing along at most $k$ vertices. Our goal is to prove the following theorem. 
\begin{theorem} \label{small cutsets} Let $k$ be a positive integer, and let $\mathcal{G}$ be a class of graphs, $\chi$-bounded by a non-decreasing function $f:\mathbb{N} \rightarrow \mathbb{R}$. Then $\mathcal{G}^k$ is $\chi$-bounded by the function $g:\mathbb{N} \rightarrow \mathbb{R}$ given by $g(n) = f(n)+2k^2-1$. 
\end{theorem} 
\noindent 
We begin with some definitions. Given a set $S$, we denote by $\mathscr{P}(S)$ the power set of $S$ (i.e.\ the set of all subsets of $S$). Given a graph $G$, we say that a four-tuple $(B,K,\phi_K,F)$ is a {\em coloring constraint} for $G$ provided that the following hold: 
\begin{itemize} 
\item $B$ is a non-empty set; 
\item $K \subseteq V_G$; 
\item $\phi_K:K \rightarrow B$ is a proper coloring of $G[K]$; 
\item $F:V_G \smallsetminus K \rightarrow \mathscr{P}(B)$. 
\end{itemize} 
$B$ should be seen as the set of colors with which we wish to color $G$, $K$ should be seen as the set of ``precolored'' vertices of $G$ with ``precoloring'' $\phi_K$, and for all $v \in V_G \smallsetminus K$, $F(v)$ should be seen a set of colors ``forbidden'' on $v$. Given a graph $G$ with a coloring constraint $(B,K,\phi_K,F)$, we say that a proper coloring $\phi:V_G \rightarrow B$ of $G$ is {\em appropriate} for $(B,K,\phi_K,F)$ provided that $\phi \upharpoonright K = \phi_K$, and that for all $v \in V_G \smallsetminus K$, $\phi(v) \notin F(v)$. We now prove a technical lemma. 
\begin{lemma} \label{small cutset constraint} Let $\mathcal{G}$ be a hereditary class, $\chi$-bounded by a non-decreasing function $f:\mathbb{N} \rightarrow \mathbb{R}$. Then for all $G \in \mathcal{G}^k$ and all coloring constraints $(B,K,\phi_K,F)$ for $G$ such that $|B| \geq f(\omega(G))+2k^2-1$ and $k|K|+\Sigma_{v \in V_G \smallsetminus K} |F(v)| \leq 2k^2-1$, there exists a proper coloring coloring $\phi:V_G \rightarrow B$ of $G$ that is appropriate for $(B,K,\phi_K,F)$. 
\end{lemma} 
\begin{proof} 
Fix $G \in \mathcal{G}^k$, and assume inductively that the claim holds for all proper induced subgraphs of $G$. Fix a coloring constraint $(B,K,\phi_K,F)$ for $G$ such that $|B| \geq f(\omega(G))+2k^2-1$ and $k|K|+\Sigma_{v \in V_G \smallsetminus K} |F(v)| \leq 2k^2-1$. We need to show that there exists a proper coloring $\phi:V_G \rightarrow B$ of $G$ that is appropriate for $(B,K,\phi_K,F)$. 
\\
\\
Suppose first that $G \in \mathcal{G}$. Since $k|K|+\Sigma_{v \in V_G \smallsetminus K} |F(v)| \leq 2k^2-1$, we know that $|\phi_K[K] \cup \bigcup_{v \in V_G \smallsetminus K} F(v)| \leq 2k^2-1$; consequently, $|B \smallsetminus (\phi_K[K] \cup \bigcup_{v \in V_G \smallsetminus K} F(v))| \leq f(\omega(G))$. Since $G \in \mathcal{G}$ and $\mathcal{G}$ is $\chi$-bounded by $f$, it follows that there exists a proper coloring $\phi':V_G \smallsetminus K \rightarrow B \smallsetminus (\phi_K[K] \cup \bigcup_{v \in V_G \smallsetminus K} F(v))$ of $G \smallsetminus K$. Now define $\phi:V_G \rightarrow B$ by setting 
\begin{displaymath} 
\begin{array}{rcl} 
\phi(v) & = & 
\left\{ \begin{array}{lcl} 
\phi_K(v) & {\rm if} & v \in K 
\\
\phi'(v) & {\rm if} & v \in V_G \smallsetminus K 
\end{array}\right. 
\end{array} 
\end{displaymath} 
By construction, the colorings $\phi_K$ and $\phi'$ use disjoint color sets; furthermore, for all $v \in V_G \smallsetminus K$, $\phi(v) \notin F(v)$. It follows that $\phi$ is a proper coloring of $G$, appropriate for $(B,K,\phi_K,F)$. 
\\
\\
Suppose now that $G \notin \mathcal{G}$. Then there exist graphs $G_1,G_2 \in \mathcal{G}^k$ with inclusion-wise incomparable vertex sets such that $G$ is obtained by gluing $G_1$ and $G_2$ along at most $k$ vertices. Set $C = V_{G_1} \cap V_{G_2}$; then $|C| \leq k$, $G_1[C] = G_2[C]$, and $G$ is obtained by gluing $G_1$ and $G_2$ along $C$. Set $V_1 = V_{G_1} \smallsetminus C$ and $V_2 = V_{G_2} \smallsetminus C$. Note that $V_G = C \cup V_1 \cup V_2$; furthermore, since the vertex sets of $G_1$ and $G_2$ are inclusion-wise incomparable, we know that $V_1$ and $V_2$ are both non-empty. By symmetry, we may assume that 
\begin{displaymath} 
\begin{array}{rcl} 
k|K \cap V_1|+\Sigma_{v \in V_1 \smallsetminus K} |F(v)| & \geq & k|K \cap V_2|+\Sigma_{v \in V_2 \smallsetminus K} |F(v)|. 
\end{array} 
\end{displaymath} 
Our first goal is to obtain a coloring constraint for $G_1$ that ``forbids'' on the vertices in $C \smallsetminus K$ all the colors used by $\phi_K$ on the set $K \cap V_2$, and then to use the induction hypothesis to obtain a coloring $\phi_1$ of $G_1$ that is appropriate for this constraint. We do this as follows. First, as $k|K|+\Sigma_{v \in V_G \smallsetminus K} |F(v)| \leq 2k^2-1$, the inequality above implies that $k|K \cap V_2|+\Sigma_{v \in V_2 \smallsetminus K} |F(v)| \leq k^2-1$, and consequently, that $|K \cap V_2| \leq k-1$. Now, set $K_1 = K \smallsetminus V_2$ and $\phi_{K_1} = \phi_K \upharpoonright K_1$. Further, define $F_1:(V_1 \cup C) \smallsetminus K \rightarrow \mathscr{P}(B)$ by setting 
\begin{displaymath} 
\begin{array}{rcl} 
F_1(v) & = & 
\left\{ \begin{array}{lcl} 
F(v) & {\rm if} & v \in V_1 \smallsetminus K_1 
\\
F(v) \cup \phi_K[K \cap V_2] & {\rm if} & v \in C \smallsetminus K_1 
\end{array}\right. 
\end{array} 
\end{displaymath} 
Clearly, $(B,K_1,\phi_{K_1},F_1)$ is a coloring constraint for $G_1$. Further, since $f$ is non-decreasing, we get that 
\begin{displaymath} 
\begin{array}{cccccc} 
|B| & \geq & f(\omega(G))+2k^2-1 & \geq & f(\omega(G_1))+2k^2-1. 
\end{array} 
\end{displaymath} 
Finally, note the following: 
\begin{displaymath} 
\begin{array}{ll} 
& k|K_1|+\Sigma_{v \in V_{G_1} \smallsetminus K_1} |F_1(v)| 
\\
\leq & k|K|-k|K \cap V_2|+\Sigma_{v \in V_{G_1} \smallsetminus K_1} |F(v)|+|C \smallsetminus K||\phi_K[K \cap V_2]| 
\\
\leq & k|K|-k|K \cap V_2|+\Sigma_{v \in V_G \smallsetminus K} |F(v)|+k|K \cap V_2| 
\\
= & k|K|+\Sigma_{v \in V_G \smallsetminus K} |F(v)| 
\\
\leq & 2k^2-1. 
\end{array} 
\end{displaymath} 
Thus, by the induction hypothesis, there exists a proper coloring $\phi_1:V_1 \cup C \rightarrow B$ of $G_1$ that is appropriate for  $(B,K_1,\phi_{K_1},F_1)$. 
\\
\\
Our next goal is to ``combine'' the coloring constraint $(B,K,\phi_K,F)$ for $G$ and the coloring $\phi_1$ of $G_1$ (or more precisely, the restriction of $\phi_1$ to $C$) in order to obtain a coloring constraint for $G_2$; we then use the induction hypothesis to obtain a coloring $\phi_2$ for $G_2$ that is appropriate for this constraint, and finally, we ``combine'' the colorings $\phi_1$ and $\phi_2$ to obtain a proper coloring $\phi$ of $G$ that is appropriate for the coloring constraint $(B,K,\phi_K,F)$. We do this as follows. First, set $K_2 = C \cup (K \cap V_2)$, and define $F_2 = F \upharpoonright (V_2 \smallsetminus K)$. Next, define $\phi_{K_2}:K_2 \rightarrow B$ by setting 
\begin{displaymath} 
\begin{array}{rcl} 
\phi_{K_2}(v) & = & 
\left\{ \begin{array}{lcl} 
\phi_1(v) & {\rm if} & v \in C 
\\
\phi_K(v) & {\rm if} & v \in K \cap V_2 
\end{array}\right. 
\end{array} 
\end{displaymath} 
Since $\phi_1$ and $\phi_K$ are proper colorings of $G_1$ and $G[K]$, respectively, and since the colors used to precolor vertices in $K \cap V_2$ were forbidden on the vertices in $C \smallsetminus K$, we get that $\phi_{K_2}$ is a proper coloring of $G_2[K_2]$. Thus, $(B,K_2,\phi_{K_2},F_2)$ is a coloring constraint for $G_2$. Since $f$ is non-decreasing, we know that 
\begin{displaymath} 
\begin{array}{cccccc} 
|B| & \geq & f(\omega(G))+2k^2-1 & \geq & f(\omega(G_2))+2k^2-1. 
\end{array} 
\end{displaymath} 
Further, since $k|K \cap V_2|+\Sigma_{v \in V_2 \smallsetminus K} |F(v)| \leq k^2-1$, we have the following: 
\begin{displaymath} 
\begin{array}{rcl} 
k|K_2|+\Sigma_{v \in V_{G_2} \smallsetminus K_2} |F_2(v)| & = & k|C|+k|K \cap V_2|+\Sigma_{v \in V_2 \smallsetminus K} |F(v)| 
\\
& \leq & k^2+k^2-1 
\\
& = & 2k^2-1. 
\end{array} 
\end{displaymath}  
Thus, by the induction hypothesis, there exists a proper coloring $\phi_2:V_{G_2} \rightarrow B$ of $G_2$ that is appropriate for $(B,K_2,\phi_{K_2},F_2)$. Note that by construction, $\phi_1 \upharpoonright C = \phi_2 \upharpoonright C$; define $\phi:V_G \rightarrow B$ by setting 
\begin{displaymath} 
\begin{array}{rcl} 
\phi(v) & = & 
\left\{ \begin{array}{lcl} 
\phi_1(v) & {\rm if} & v \in V_1 \cup C  
\\
\phi_2(v) & {\rm if} & v \in V_2 \cup C
\end{array}\right. 
\end{array} 
\end{displaymath} 
By construction, $\phi$ is a proper coloring of $G$, appropriate for $(B,K,\phi_K,F)$. This completes the argument. 
\end{proof} 
\noindent 
We are now ready to prove Theorem \ref{small cutsets}, restated below. 
\begin{small cutsets} Let $k$ be a positive integer, and let $\mathcal{G}$ be a class of graphs, $\chi$-bounded by a non-decreasing function $f:\mathbb{N} \rightarrow \mathbb{R}$. Then $\mathcal{G}^k$ is $\chi$-bounded by the function $g:\mathbb{N} \rightarrow \mathbb{R}$ given by $g(n) = f(n)+2k^2-1$. 
\end{small cutsets} 
\begin{proof} 
We may assume that $\mathcal{G}$ is hereditary (otherwise, instead of $\mathcal{G}$, we consider the closure of $\mathcal{G}$ under isomorphism and taking induced subgraphs). Fix $G \in \mathcal{G}^k$. Let $B = \{1,...,f(\omega(G))+2k^2-1\}$, let $K = \emptyset$, let $\phi_K$ be the empty function, and define $F:V_G \rightarrow \mathscr{P}(B)$ by setting $F(v) = \emptyset$ for all $v \in V_G$. Then $(B,K,\phi_K,F)$ is a coloring constraint for $G$ with $|B| \geq f(\omega(G))+2k^2-1$ and $k|K|+\Sigma_{v \in V_G \smallsetminus K} |F(v)| \leq 2k^2-1$. By Lemma \ref{small cutset constraint} then, there exists a proper coloring $\phi:V_G \rightarrow B$ that is appropriate for $(B,K,\phi_K,F)$. But then $\phi$ is a proper coloring of $G$ that uses at most $g(\omega(G))$ colors. 
\end{proof} 
\noindent 
As remarked in the Introduction, the fact that the closure of a $\chi$-bounded class is again $\chi$-bounded follows from a result proven in \cite{alon}. It was proven in \cite{alon} that every graph of chromatic number greater than $\max\{100k^3,m+10k^2\}$ contains a $(k+1)$-connected subgraph of chromatic number at least $m$. It is easy to see that this result implies that if $\mathcal{G}$ is $\chi$-bounded by a non-decreasing function $f:\mathbb{N} \rightarrow \mathbb{R}$, then $\mathcal{G}^k$ is $\chi$-bounded by the function $g:\mathbb{N} \rightarrow \mathbb{R}$ defined by $g(n) = \max\{100k^3,f(n)+10k^2+1\}$. Note, however, that the $\chi$-bounding function from Theorem \ref{small cutsets} is better than the $\chi$-bounding function that follows from the result of \cite{alon}. Conversely, Theorem \ref{small cutsets} implies the following strengthening of the theorem from \cite{alon}. 
\begin{cor} \label{high-conn} Let $m,k \in \mathbb{Z}^+$. Then every graph of chromatic number greater than $\max\{2k^2+k,m+2k^2-1\}$ contains a $(k+1)$-connected induced subgraph of chromatic number at least $m$. 
\end{cor} 
\begin{proof} 
Let $G$ be a graph such that $\chi(G) > \max\{2k^2+k,m+2k^2-1\}$. Then $\chi(G) \geq \max\{k+2,m\}+2k^2-1$. Next, let $\mathcal{G}$ be the class of all graphs whose chromatic number is at most $\max\{k+2,m\}-1$. Clearly, $\mathcal{G}$ is $\chi$-bounded by the constant function $f:\mathbb{N} \rightarrow \mathbb{R}$ given by $f(n) = \max\{k+2,m\}-1$. By Theorem \ref{small cutsets}, we know that $\mathcal{G}^k$ is $\chi$-bounded by the constant function $g:\mathbb{N} \rightarrow \mathbb{R}$ given by $g(n) = \max\{k+2,m\}+2k^2-2$. Since $\chi(G) \geq \max\{k+2,m\}+2k^2-1$, it follows that $G \notin \mathcal{G}^k$. This implies that there exists an induced subgraph $H$ of $G$ such that either $H$ is a complete graph or $H$ is $(k+1)$-connected, and such that $H \notin \mathcal{G}$.  
\\
\\
Suppose first that $H$ is a complete graph. Then $\chi(H) = |V_H|$, and so since $H \notin \mathcal{G}$, it follows that $|V_H| \geq \max\{k+2,m\}$. But then since $H$ is a complete graph, it follows that $H$ is $(k+1)$-connected and that $\chi(H) \geq m$, and we are done. 
\\
\\
Suppose now that $H$ is $(k+1)$-connected. But then since $H \notin \mathcal{G}$, we have that $\chi(H) \geq m$, and again we are done. 
\end{proof} 
\noindent 
We complete this subsection by considering ``combinations'' of gluing along a clique and gluing along a bounded number of vertices. Given a class $\mathcal{G}$ of graphs and a positive integer $k$, we denote by $\mathcal{G}^k_{cl}$ the closure of $\mathcal{G}$ under gluing along a clique and gluing along at most $k$ vertices. Our goal is to prove that if $\mathcal{G}$ is a $\chi$-bounded class, then for every $k \in \mathbb{Z}^+$, the class $\mathcal{G}^k_{cl}$ is $\chi$-bounded (see Proposition \ref{gluing combination} below). We begin with a technical lemma, which we then use to prove Proposition \ref{gluing combination}. 
\begin{lemma} \label{switch order} Let $\mathcal{G}$ be a hereditary class, closed under gluing along a clique, and let $k$ be a positive integer. Then $\mathcal{G}^k$ is closed under gluing along a clique, and consequently, $\mathcal{G}^k = \mathcal{G}^k_{cl}$. 
\end{lemma} 
\begin{proof} 
Clearly, the second claim follows from the first, and so we just need to show that $\mathcal{G}^k$ is closed under gluing along a clique. Let $\tilde{\mathcal{G}}^k$ be the closure of $\mathcal{G}^k$ under gluing along a clique. We claim that $\tilde{\mathcal{G}}^k = \mathcal{G}^k$. Fix $G \in \tilde{\mathcal{G}}^k$, and assume inductively that for all $G' \in \tilde{\mathcal{G}}^k$ such that $|V_{G'}| < |V_G|$, we have that $G' \in \mathcal{G}^k$; we claim that $G \in \mathcal{G}^k$. 
\\
\\
By the definition of $\tilde{\mathcal{G}}^k$, we know that either $G \in \mathcal{G}^k$, or $G$ is obtained by gluing smaller graphs in $\tilde{\mathcal{G}}^k$ along a clique. In the former case, we are done; so assume that there exist graphs $G_1,G_2 \in \tilde{\mathcal{G}}^k$ such that $|V_{G_i}| < |V_G|$ for each $i \in \{1,2\}$, such that $C = V_{G_1} \cap V_{G_2}$ is a clique in both $G_1$ and $G_2$, and such that $G$ is obtained by gluing $G_1$ and $G_2$ along the clique $C$. By the induction hypothesis, $G_1,G_2 \in \mathcal{G}^k$. Now, if $G_1,G_2 \in \mathcal{G}$, then the fact that $\mathcal{G}$ is closed under gluing along a clique implies that $G \in \mathcal{G}$, and consequently, that $G \in \mathcal{G}^k$. So assume that at least one of $G_1$ and $G_2$ is not a member of $\mathcal{G}$; by symmetry, we may assume that $G_1 \notin \mathcal{G}$. 
\\
\\
Since $G_1 \in \mathcal{G}^k \smallsetminus \mathcal{G}$, there exist graphs $G_1^1,G_1^2 \in \mathcal{G}^k$ such that $|V_{G_1^i}| < |V_{G_1}|$ for each $i \in \{1,2\}$, and such that $G_1$ is obtained by gluing $G_1^1$ and $G_1^2$ along $K = V_{G_1^1} \cap V_{G_1^2}$, where $|K| \leq k$. Now, $C$ is a clique in $G_1$, and so we know that $C \subseteq V_{G_1^i}$ for some $i \in \{1,2\}$; by symmetry, we may assume that $C \subseteq V_{G_1^1}$. If $C = V_{G_1^1}$, then set $G_1' = G_2$; and if $C \subsetneqq V_{G_1^1}$, then let $G_1'$ be the graph obtained by gluing $G_1^1$ and $G_2$ along $C$. As $G_1^1,G_2 \in \mathcal{G}^k$, we know that $G_1' \in \tilde{\mathcal{G}}^k$. Further, note that $|V_{G_1'}| < |V_G|$, and so by the induction hypothesis, $G_1' \in \mathcal{G}^k$. But now $G$ is obtained by gluing $G_1'$ and $G_1^2$ along $K$, and so since $G_1',G_1^2 \in \mathcal{G}^k$, we know that $G \in \mathcal{G}^k$. This completes the argument. 
\end{proof} 
\begin{prop} \label{gluing combination} Let $\mathcal{G}$ be a class of graphs, $\chi$-bounded by a non-decreasing function $f:\mathbb{N} \rightarrow \mathbb{R}$, and let $k$ be a positive integer. Then $\mathcal{G}^k_{cl}$ is $\chi$-bounded by the function $g:\mathbb{N} \rightarrow \mathbb{R}$ given by $g(n) = f(n)+2k^2-1$. 
\end{prop} 
\begin{proof} 
We may assume that $\mathcal{G}$ is hereditary (otherwise, instead of $\mathcal{G}$, we consider the closure of $\mathcal{G}$ under isomorphism and taking induced subgraphs). Let $\tilde{\mathcal{G}}$ be the closure of $\mathcal{G}$ under gluing along a clique. Then by Lemma \ref{switch order}, $\tilde{\mathcal{G}}^k = \mathcal{G}^k_{cl}$ (where $\tilde{\mathcal{G}}^k$ is the closure of $\tilde{\mathcal{G}}$ under gluing along at most $k$ vertices). By Proposition \ref{gluing clique only}, $\tilde{\mathcal{G}}$ is $\chi$-bounded by $f$; but then by Theorem \ref{small cutsets}, $\tilde{\mathcal{G}}^k$ is $\chi$-bounded by $g$. It follows that $\mathcal{G}^k_{cl}$ is $\chi$-bounded by $g$. 
\end{proof}

\subsection{Substitution and Gluing along a Clique} 

In section \ref{section:substitution}, we proved that the closure of a $\chi$-bounded class under substitution is $\chi$-bounded (see Theorem \ref{general bound}, as well as Theorem \ref{poly} and Theorem \ref{exp} for a strengthening of Theorem \ref{general bound} in some special cases), and in this section, we proved an analogous result for gluing along a clique (see Proposition \ref{gluing clique only}). In this subsection, we discuss ``combinations'' of these two operations. Given a class $\mathcal{G}$ of graphs, we denote by $\mathcal{G}^\#$ the closure of $\mathcal{G}$ under substitution and gluing along a clique. Our main goal is to prove the following proposition. 
\begin{prop} \label{closure} Let $\mathcal{G}$ be a class of graphs, $\chi$-bounded by a non-decreasing function $f:\mathbb{N} \rightarrow \mathbb{R}$. Then $\mathcal{G}^\#$ is $\chi$-bounded by the function $g(k) = f(k)^k$. 
\end{prop} 
\noindent 
We note that, as in section \ref{section:substitution}, we can obtain a strengthening of Proposition \ref{closure} in the case when the $\chi$-bounding function for the class $\mathcal{G}$ is polynomial or exponential (see Proposition \ref{gluing polyexp}). The main ``ingredient'' in the proof of Proposition \ref{closure} is the following lemma. 
\begin{lemma} \label{gluing bound} Let $\mathcal{G}$ be a hereditary class, closed under substitution. Assume that $\mathcal{G}$ is $\chi$-bounded by a non-decreasing function $f:\mathbb{N} \rightarrow \mathbb{R}$. Then $\mathcal{G}^\#$ is $\chi$-bounded by $f$. 
\end{lemma} 
\noindent 
In view of the results of section \ref{section:substitution}, Lemma \ref{gluing bound} easily implies Proposition \ref{closure} and Proposition \ref{gluing polyexp} (see the proof of these two theorems at the end of this section). The idea of the proof of Lemma \ref{gluing bound} is as follows. We first prove a certain structural result for graphs in the class $\mathcal{G}^\#$, where $\mathcal{G}$ is a hereditary class, closed under substitution (see Lemma \ref{decomposition}). We then use Lemma \ref{decomposition} to show that if $\mathcal{G}$ is a hereditary class, closed under substitution, then for every graph $G \in \mathcal{G}^\#$, there exists a graph $G' \in \mathcal{G}$ such that $G'$ is an induced subgraph of $G$ and $\chi(G') = \chi(G)$ (see Lemma \ref{gluing reduction}). Finally, Lemma \ref{gluing reduction} easily implies Lemma \ref{gluing bound}. 
\\
\\
We begin with some definitions. Let $\mathcal{G}$ be a hereditary class. Given non-empty graphs $G,G_0 \in \mathcal{G}^\#$ with $V_{G_0} = \{v_1,...,v_t\}$, we say that $G$ is an {\em expansion} of $G_0$ provided that there exist non-empty graphs $G_1,...,G_t \in \mathcal{G}^\#$ with pairwise disjoint vertex sets such that $G$ is obtained by substituting $G_1,...,G_t$ for $v_1,..,v_t$ in $G_0$. We observe that every non-empty graph in $\mathcal{G}^\#$ is an expansion of itself. We say that a non-empty graph $G \in \mathcal{G}^\#$ is {\em decomposable} provided that there exists a non-empty graph $G' \in \mathcal{G}^\#$ such that $G$ is an expansion of $G'$, and there exist non-empty graphs $H,K \in \mathcal{G}^\#$ with inclusion-wise incomparable vertex sets such that $G'$ can be obtained by gluing $H$ and $K$ along a clique. We now prove a structural result for graphs in $\mathcal{G}^\#$, when $\mathcal{G}$ is a hereditary class, closed under substitution.  
\begin{lemma} \label{decomposition} Let $\mathcal{G}$ be a hereditary class, closed under substitution. Then for every graph $G \in \mathcal{G}^\#$, either $G \in \mathcal{G}$, or there exists a non-empty set $S \subseteq V_G$ such that $S$ is a homogeneous set in $G$ and $G[S]$ is decomposable. 
\end{lemma} 
\begin{proof} 
Let $G \in \mathcal{G}^\#$, and assume inductively that the claim holds for every graph in $\mathcal{G}^\#$ with fewer than $|V_G|$ vertices. If $G \in \mathcal{G}$, then we are done. So assume that $G \in \mathcal{G}^\# \smallsetminus \mathcal{G}$. If $G$ can be obtained from two graphs in $\mathcal{G}^\#$, each with fewer than $|V_G|$ vertices, by gluing along a clique, then $G$ is decomposable, and we are done. So assume that this is not the case. Then there exist non-empty graphs $G_1,G_2 \in \mathcal{G}^\#$ such that $V_{G_1} \cap V_{G_2} = \emptyset$ and $|V_{G_i}| < |V_G|$ for each $i \in \{1,2\}$, and a vertex $u \in V_{G_1}$, such that $G$ is obtained by substituting $G_2$ for $u$ in $G_1$. 
\\
\\
By the induction hypothesis, either $G_2 \in \mathcal{G}$ or there exists a homogeneous set $S_2 \subseteq V_{G_2}$ in $G_2$ such that $G_2[S_2]$ is decomposable. In the latter case, it easy to see that the set $S_2$ is a homogeneous set in $G$ as well, and that $G[S_2]$ is decomposable. So from now on, we assume that $G_2 \in \mathcal{G}$. Now, if $G_1 \in \mathcal{G}$, then since $G_2 \in \mathcal{G}$ and $\mathcal{G}$ is closed under substitution, we get that $G \in \mathcal{G}$, which is a contradiction. Thus, $G_1 \notin \mathcal{G}$. By the induction hypothesis then, there exists a non-empty set $S_1 \subseteq V_{G_1}$ such that $S_1$ is a homogeneous set in $G_1$ and $G_1[S_1]$ is decomposable. If $u \notin S_1$, then it is easy to see that $S_1$ is a homogeneous set in $G$ and that $G[S_1]$ is decomposable. So assume that $u \in S_1$. Set $S = (S_1 \smallsetminus \{u\}) \cup V_{G_2}$. Clearly, $S$ is a homogeneous set in $G$ (as $S_1$ is a homogeneous set in $G_1$). Further, $G[S]$ is obtained by substituting $G_2$ for $u$ in the decomposable graph $G_1[S_1]$, and so it is easy to see that $G[S]$ is decomposable. This completes the argument. 
\end{proof} 
\begin{lemma} \label{gluing reduction} Let $\mathcal{G}$ be a hereditary class, closed under substitution. Then for all $G \in \mathcal{G}^\#$, there exists a graph $G' \in \mathcal{G}$ such that $G'$ is an induced subgraph of $G$ and $\chi(G') = \chi(G)$. 
\end{lemma} 
\begin{proof} 
Fix a graph $G \in \mathcal{G}^\#$, and assume inductively that the claim holds for every graph in $\mathcal{G}^\#$ that has fewer than $|V_G|$ vertices. If $G \in \mathcal{G}$, then the result is immediate; so assume that $G \notin \mathcal{G}$. Then by Lemma \ref{decomposition}, there exists a non-empty set $S \subseteq V_G$ such that $S$ is a homogeneous set in $G$ and $G[S]$ is decomposable. 
\\
\\
Since $G[S]$ is decomposable, there exist graphs $G_0,H_0,K_0 \in \mathcal{G}^\#$ such that $H_0$ and $K_0$ have inclusion-wise incomparable vertex sets, such that $G_0$ can be obtained by gluing $H_0$ and $K_0$ along a clique, and such that $G[S]$ is an expansion of $G_0$. Set $C = V_{H_0} \cap V_{K_0}$; then $C$ is a clique in both $H_0$ and $K_0$, and $G_0$ is obtained by gluing $H_0$ and $K_0$ along $C$. Set $C = \{c_1,...,c_r\}$, $V_{H_0} \smallsetminus C = \{h_1,...,h_s\}$, and $V_{K_0} \smallsetminus C = \{k_1,...,k_t\}$. Let $C_1,...,C_r,H_1,...,H_s,K_1,...,K_t$ be non-empty graphs with pairwise disjoint vertex sets such that $G[S]$ is obtained by substituting $C_1,...,C_r,H_1,...,H_s,K_1,...,K_s$ for $c_1,...,c_r,h_1,...,h_s,k_1,...,k_t$ in $G_0$. Set $\tilde{C} = \bigcup_{i=1}^r V_{C_i}$. Let $H$ be the graph obtained by substituting $C_1,...,C_r,H_1,...,H_s$ for $c_1,...,c_r,h_1,...,h_s$ in $H_0$; and let $K$ be the graph obtained by substituting $C_1,...,C_r,K_1,...,K_t$ for $c_1,...,c_r,k_1,...,k_t$ in $K_0$. Clearly, both $H$ and $K$ are proper induced subgraphs of $G[S]$. Our goal is to show that $\chi(G[S]) = \max\{\chi(H),\chi(K)\}$. Since $H$ and $K$ are induced subgraphs of $G[S]$, it suffices to show that $\chi(G[S]) \leq \max\{\chi(H),\chi(K)\}$. 
\\
\\
Let $b_H':V_H \rightarrow \{1,...,\chi(H)\}$ be an optimal coloring of $H$. Since $V_{C_1},...,V_{C_r}$ are pairwise disjoint and complete to each other, we know that $b_H'$ uses pairwise disjoint color sets on these sets. Now, let $b_H:V_H \rightarrow \{1,...,\chi(H)\}$ be defined as follows: for all $v \in V_H \smallsetminus \tilde{C}$, set $b_H(v) = b_H'(v)$, and for all $i \in \{1,...,r\}$, assume that $b_H \upharpoonright V_{C_i}$ is an optimal coloring of $C_i$ using only the colors from $b_H'[C_i]$. As $V_{C_1},...,V_{C_r}$ are homogeneous sets in $H$, and $b_H[C_i] \subseteq b_H'[C_i]$ for all $i \in \{1,...,r\}$, it easily follows that $b_H$ is a proper coloring of $H$. Now, note that $b_H:V_H \rightarrow \{1,...,\chi(H)\}$ is an optimal coloring of $H$, $b_H[V_{C_1}],...,b_H[V_{C_r}]$ are pairwise disjoint, and for each $i \in \{1,...,r\}$, $|b_H[V_{C_i}]| = \chi(C_i)$. Similarly, there exists an optimal coloring $b_K:V_K \rightarrow \{1,...,\chi(K)\}$ of $K$ such that $b_K[V_{C_1}],...,b_K[V_{C_r}]$ are pairwise disjoint, and for each $i \in \{1,...,r\}$, $|b_K[V_{C_i}]| = \chi(C_i)$. Relabeling if necessary, we may assume that $b_H \upharpoonright \tilde{C} = b_K \upharpoonright \tilde{C}$; as $V_H \cap V_K = \tilde{C}$, we can define $b_S:S \rightarrow \{1,...,\max\{\chi(H),\chi(K)\}\}$ by setting 
\begin{displaymath} 
\begin{array}{rcl} 
b_S(v) & = & 
\left\{ \begin{array}{lcl} 
b_H(v) & {\rm if} & v \in V_H 
\\
b_K(v) & {\rm if} & v \in V_K 
\end{array}\right. 
\end{array} 
\end{displaymath} 
Since $V_H \smallsetminus \tilde{C}$ is anti-complete to $V_K \smallsetminus \tilde{C}$ in $G[S]$, this is a proper coloring of $G[S]$. It follows that $\chi(G[S]) = \max\{\chi(H),\chi(K)\}$. By symmetry, we may assume that $\chi(K) \leq \chi(H)$, so that $\chi(G[S]) = \chi(H)$. 
\\
\\
Now, since $S$ is a homogeneous set in $G$, there exists a graph $\tilde{G} \in \mathcal{G}^\#$ such that $V_{\tilde{G}} \cap S = \emptyset$, and a vertex $u \in V_{\tilde{G}}$ such that $G$ is obtained by substituting $G[S]$ for $u$ in $\tilde{G}$. Let $G_H$ be the graph obtained by substituting $H$ for $u$ in $\tilde{G}$. Since $\chi(G[S]) = \chi(H)$, it is easy to see that $\chi(G[S]) = \chi(G_H)$. Since $H$ is a proper induced subgraph of $G[S]$, we have that $G_H$ is a proper induced subgraph of $G$. By the induction hypothesis then, there exists a graph $G' \in \mathcal{G}$ such that $G'$ is an induced subgraph of $G_H$ and $\chi(G') = \chi(G_H)$. But then $G' \in \mathcal{G}$ is an induced subgraph of $G$ and $\chi(G') = \chi(G)$. This completes the argument. 
\end{proof} 
\noindent 
We are now ready to prove Lemma \ref{gluing bound}, restated below. 
\begin{gluing bound} Let $\mathcal{G}$ be a hereditary class, closed under substitution. Assume that $\mathcal{G}$ is $\chi$-bounded by a non-decreasing function $f:\mathbb{N} \rightarrow \mathbb{R}$. Then $\mathcal{G}^\#$ is $\chi$-bounded by $f$. 
\end{gluing bound} 
\begin{proof} 
Fix $G \in \mathcal{G}^\#$. By Lemma \ref{gluing reduction}, there exists a graph $G' \in \mathcal{G}$ such that $G'$ is an induced subgraph of $G$ and $\chi(G') = \chi(G)$. Since $G'$ is an induced subgraph of $G$, we know that $\omega(G') \leq \omega(G)$. Since $G' \in \mathcal{G}$ and $\mathcal{G}$ is $\chi$-bounded by $f$, we have that $\chi(G') \leq f(\omega(G'))$. Finally, since $\omega(G') \leq \omega(G)$ and $f$ is non-decreasing, $f(\omega(G')) \leq f(\omega(G))$. Now we have the following: 
\begin{displaymath} 
\begin{array}{ccccccc} 
\chi(G) & = & \chi(G') & \leq & f(\omega(G')) & \leq & f(\omega(G)). 
\end{array} 
\end{displaymath} 
It follows that $\mathcal{G}^\#$ is $\chi$-bounded by $f$. 
\end{proof} 
\noindent 
We now use Lemma \ref{gluing bound} and the results of section \ref{section:substitution}, in order to prove Proposition \ref{closure} (restated below), as well as Proposition \ref{gluing polyexp}, which is a strengthening of Proposition \ref{closure} in some special cases. 
\begin{closure} Let $\mathcal{G}$ be a class of graphs, $\chi$-bounded by a non-decreasing function $f:\mathbb{N} \rightarrow \mathbb{R}$. Then $\mathcal{G}^\#$ is $\chi$-bounded by the function $g(k) = f(k)^k$. 
\end{closure} 
\begin{proof} 
We may assume that $\mathcal{G}$ is hereditary (otherwise, instead of $\mathcal{G}$, we consider the closure of $\mathcal{G}$ under isomorphism and taking induced subgraphs). Now, if $\mathcal{G}$ contains no non-empty graphs, then neither does $\mathcal{G}^\#$, and then $\mathcal{G}^\#$ is $\chi$-bounded by $g$ because $g(0) = 1$ and $\chi(H) = \omega(H) = 0$ for the empty graph $H$. So assume that $\mathcal{G}$ contains at least one non-empty graph; this implies that $f(0) \geq 0$ and that $f(n) \geq 1$ for all $n \in \mathbb{Z}^+$; since $f$ is non-decreasing, this implies that $g$ is non-decreasing. Now, by Theorem \ref{general bound}, $\mathcal{G}^*$ is $\chi$-bounded by $g$. Next, note that $\mathcal{G}^\#$ is the closure of $\mathcal{G}^*$ under substitution and gluing along a clique. Since $\mathcal{G}^*$ closed under substitution and is $\chi$-bounded by the non-decreasing function $g$, Lemma \ref{gluing bound} implies that $\mathcal{G}^\#$ is $\chi$-bounded by $g$. 
\end{proof} 
\begin{prop} \label{gluing polyexp} Let $\mathcal{G}$ be a class of graphs, $\chi$-bounded by a polynomial (respectively: exponential) function $f:\mathbb{N} \rightarrow \mathbb{R}$. Then $\mathcal{G}^\#$ is $\chi$-bounded by some polynomial (respectively: exponential) function $g:\mathbb{N} \rightarrow \mathbb{R}$. 
\end{prop} 
\begin{proof} 
The proof is analogous to the proof of Proposition \ref{closure}, with Theorem \ref{poly} and Theorem \ref{exp} being used instead of Theorem \ref{general bound}. 
\end{proof}

\section{Open Questions} \label{section:questions}

Let us say that an operation $O$ defined on the class of graphs {\em preserves $\chi$-boundedness} (respectively: {\em preserves hereditariness}) if for for every $\chi$-bounded (respectively: hereditary) class $\mathcal{G}$, the closure of $\mathcal{G}$ under the operation $O$ is again $\chi$-bounded (respectively: hereditary). This work raises the following natural question. Suppose that some $\chi$-boundedness preserving operations are given. Is the closure of a $\chi$-bounded class with respect to all the operations together $\chi$-bounded? In general, the answer is no. The \emph{Mycielskian} $M(G)$ of a graph $G$ on $\{v_1,\dots,v_n\}$ is defined as follows: start from $G$ and for all $i \in \{1,...,n\}$, add a vertex $w_i$ complete to $N_G(v_i)$ (note that $\{w_1,\dots,w_n\}$ is a stable set in $M(G)$); then add a vertex $w$ complete to $\{w_1,\dots,w_n\}$. It is well known (see~\cite{mycielski:color}) that $\omega(M(G)) = \omega(G)$ and $\chi(M(G)) = \chi(G)+1$ for every graph $G$ that has at least one edge. Now define two operations on graphs: $O_1(G)$ (respectively: $O_2(G)$) is defined to be $M(G)$ if $\chi(G)$ is odd (respectively: even), and to be $G$ otherwise. Clearly, $O_1$ (respectively: $O_2$) preserves $\chi$-boundedness; this follows from the fact that applying $O_1$ (respectively: $O_2$) repeatedly can increase the chromatic number of a graph at most by $1$.  But taken together, $O_1$ and $O_2$ may build triangle-free graphs of arbitrarily large chromatic number: by applying them alternately to the complete graph on two vertices for instance. However, this example may look artificial; perhaps some more ``natural'' kinds of operations, to be defined, have better behavior? Note that, unlike the three operations discussed in this paper (substitution, gluing along a clique, and gluing along a bounded number of vertices), the Mycielskian does not preserve hereditariness. This suggests a candidate for which we have no counterexample:
\begin{question} If $O_1$ and $O_2$ are operations that (individually) preserve hereditariness and $\chi$-boundedness, do $O_1$ and $O_2$ together preserve $\chi$-boundedness? 
\end{question} 
\noindent 
Note that we do not know the answer in the following particular case:
\begin{question} 
Is the closure of a $\chi$-bounded class under substitution and gluing along a bounded number of vertices $\chi$-bounded?
\end{question} 
\noindent 
Are there other operations that in some sense preserve $\chi$-boundedness?  A \emph{star} in a graph $G$ is a set $S \subseteq V_G$ such that some vertex $v \in S$ is complete to $S \smallsetminus \{v\}$.  A \emph{star cutset} of a graph is star whose deletion yields a disconnected graph. Star cutsets are interesting in our context because their introduction by Chv\'atal~\cite{chvatal:starcutset} was the first step in a series of theorems that culminated in the proof of the strong perfect graph conjecture. Also, several classes of graphs that are notoriously difficult to decompose are decomposed with star cutsets or variations on star cutsets: star cutsets are used to decompose even-hole-free graphs (see \cite{dsv:ehf}); skew partitions are used to decompose Berge graphs (see \cite{SPGT}); double star cutsets are used to decompose odd-hole-free graphs (see \cite{conforti.c.v:dstrarcut}). Could it be that some of these decompositions preserve $\chi$-boundedness? If so, the following open question could be a good starting point (and should have a positive answer):
\begin{question} Is there a constant $c$ such that if a graph $G$ is triangle-free and all induced subgraphs of $G$ either are 3-colorable or admit a star cutset, then $G$ is $c$-colorable? 
\end{question}

\end{document}